\documentclass[11pt]{article}

\usepackage[margin=1in]{geometry}
\usepackage[T1]{fontenc}
\usepackage{lmodern}
\usepackage{amsmath,amssymb,amsfonts,amsthm,mathtools}
\usepackage{bm}
\usepackage{microtype}
\usepackage{enumitem}
\usepackage[numbers,sort&compress]{natbib}
\usepackage{url}
\usepackage[hidelinks]{hyperref}

\setlist[enumerate]{leftmargin=*,itemsep=0.45em,topsep=0.5em}

\newcommand{\R}{\mathbb{R}}
\newcommand{\norm}[1]{\left\lVert #1\right\rVert}
\newcommand{\ip}[2]{\left\langle #1,#2\right\rangle}
\newcommand{\cF}{\mathcal{F}}
\newcommand{\prox}{\operatorname{prox}}
\newcommand{\doi}[1]{\url{https://doi.org/#1}}

\newtheorem{theorem}{Theorem}
\newtheorem{proposition}{Proposition}
\newtheorem{lemma}{Lemma}
\theoremstyle{definition}
\newtheorem{definition}{Definition}

\title{Last-Iterate Performance of Gradient Descent and Relaxed Proximal Point via $s$-Composability}
\author{Salah R. Chikhi\\
\small \'Ecole Polytechnique}
\date{September 2026}

\begin{document}
\maketitle

\begin{abstract}
The notion of $s$-composability was introduced to compose optimized stepsize
schedules while preserving sharp guarantees. We show that the same joint
potential has a second use: it can serve directly as a certificate of sharp
last-iterate performance. We develop this viewpoint for constant and silver
schedules. For the constant schedule, we prove the $s$-composability statement
posed as an open question in \cite{GrimmerShuWang2025}, at every finite
horizon, through an explicit nonnegative smooth-convex interpolation
certificate. The result identifies the unique constant stepsize minimizing the
worst-case final gradient norm of smooth convex gradient descent under an
initial-distance bound, thereby proving the optimal stepsize, value, and
uniqueness predicted by Conjecture~3 in \cite{TaylorHendrickxGlineur2017}
without resolving its full worst-case curve. Through the Moreau envelope, the
same constant is also the unique constant relaxation minimizing the worst-case
final residual of relaxed proximal point. More generally, for every positive
$s$-composable schedule, we determine the exact worst-case final proximal
objective-gap constant and give a matching one-dimensional example. Applying
this result to the already $s$-composable original silver schedule yields its
exact last-iterate objective-gap constant, closing a question left open by
\cite{WangMaYangZhou2025}.
\end{abstract}

\noindent\textbf{Keywords:} gradient descent; proximal point method; performance estimation
\medskip

\section{Introduction}
\label{sec:introduction}

The performance of the last iterate is a basic but often delicate question in
first-order optimization.  Even for gradient descent on a smooth convex
function, the exact behavior of the terminal gradient or objective value can
depend sharply on the choice of stepsizes, and tight guarantees are known only
for particular criteria and parameter regimes.  Similar issues arise for
proximal methods, where objective accuracy and stationarity of the terminal
proximal point need not be governed by the same worst-case mechanism.  Recent
work on nonstandard stepsize schedules has shown that carefully chosen long
steps can substantially improve worst-case guarantees, making the exact
terminal behavior of such schedules a natural question in its own right.

A central object in this recent literature is \(s\)-composability, introduced
in Definition~3 of \cite{GrimmerShuWang2025}. Its original purpose is
compositional: an \(s\)-composable schedule carries a joint potential that
allows optimized schedules to be concatenated while preserving sharp rates.
The same potential simultaneously controls the terminal objective gap,
gradient norm, and distance to the solution set, so it can also be used
directly for last-iterate analysis. We develop this viewpoint on two schedules:
a constant schedule, whose \(s\)-composability was left open by
\cite{GrimmerShuWang2025}, and the silver schedule, for which composability is
already known.

\subsection{Related work and open questions}
\label{subsec:related}
Our results connect three strands of exact worst-case analysis that have so far
been studied largely separately.

\paragraph{Final-gradient performance of constant-step gradient descent.}
Consider \(n\) iterations of gradient descent with a constant normalized
stepsize \(\alpha\), and suppose only that the initial distance
\(\norm{x_0-x_*}\) is bounded.  A basic unresolved question is the exact
worst-case value of the final gradient norm \(\norm{\nabla f(x_n)}\).
Performance-estimation methods \cite{DroriTeboulle2014,TaylorHendrickxGlineur2017} suggest that, for
\(0<\alpha\leq2\), the answer is governed by the larger of two simple
obstructions:

$$
    \frac{1}{1+n\alpha}
    \qquad\text{and}\qquad
    |1-\alpha|^n.
$$

The first comes from instances on which gradient descent makes only limited
progress, while the second comes from oscillation on a quadratic objective.
Conjecture~3 in \cite{TaylorHendrickxGlineur2017}, specialized
to smooth convex functions, predicts that these two examples describe the
entire worst-case curve.  That conjecture is still open.  Importantly for this
paper, however, the two branches intersect at a unique stepsize, so the
conjecture already predicts a concrete candidate for the \emph{best constant
stepsize}.  We prove that this candidate is indeed globally optimal and unique,
without needing to establish the full conjectured curve.

Several exact results nearby concern different performance measures.
For example, \cite{Kim2024} determines the final objective-gap behavior and
studies the related final-gradient problem under a function-gap rather than a
distance normalization.  \cite{RotaruGlineurPatrinos2026} studies the minimum
gradient norm attained \emph{along} a constant-step trajectory, again under a
different normalization.  More recently, \cite{Zhang2026} establishes the
certificate needed for the constant stepsize arising in the classical
minimax objective-gap problem.  None of these results determines the optimal
constant stepsize for the final-gradient, initial-distance criterion considered
here.

\paragraph{\(s\)-composability and constant schedules.}
A different line of work asks how optimized stepsize schedules can be combined
while retaining sharp worst-case guarantees.  Motivated in part by the
accelerated ``silver'' schedules \cite{AltschulerParrilo2025},
\(s\)-composability was introduced in \cite{GrimmerShuWang2025}: a schedule is
\(s\)-composable when it satisfies a joint potential inequality strong enough
to support such compositions.  For a constant schedule of length \(n\), the
algebraic conditions required by \(s\)-composability determine a unique
candidate stepsize.  Example~2 of \cite{GrimmerShuWang2025} verified the first
two horizons but left open whether this candidate is actually
\(s\)-composable for every \(n\geq3\).  We establish this statement for all
finite horizons.  The key observation for our purposes is that the resulting joint
potential also gives exactly the terminal-gradient control needed to certify
the optimal constant stepsize described above.

\paragraph{Last-iterate guarantees for relaxed proximal point.}
There is a closely related proximal problem.  Through the Moreau envelope,
relaxed proximal point can be represented as gradient descent on a smooth
convex function, so stepsize results can often be transferred between the two
settings.  For constant relaxation parameters,
Theorem~3.1 and Remark~3.1 of \cite{WangMaYangZhou2025} derive tight
last-iterate objective and residual guarantees in their stated range.  We show
that, when the final residual is optimized over \emph{all} positive constant
relaxations, the unique optimizer is exactly the same constant as in the
smooth final-gradient problem.

The same authors also analyze the original silver schedule.  Its
\(s\)-composability was already known from \cite{GrimmerShuWang2025}, and
Theorem~5.3 of \cite{WangMaYangZhou2025} gives a tight bound for its final
proximal residual.  Their corresponding objective-gap bound, however, is
explicitly noted to be non-tight, leaving its exact last-iterate value open.
We determine that value exactly.

\subsection{Contributions}
\label{subsec:contributions}
We make three contributions.

\begin{enumerate}
\item \textbf{Constant-step \(s\)-composability.}
For every \(n\geq1\), we prove that the unique constant schedule compatible
with the two \(s\)-composability rate identities is indeed \(s\)-composable.
Writing its step as \(h_n\), the defining balance is
\[
    \frac{1}{1+n h_n}=(h_n-1)^n.
\]
This proves the statement posed as an open question in
Example~2 of \cite{GrimmerShuWang2025}. 

\item \textbf{Exact constant-parameter selection for the last iterate.}
We show, through $s$-composability, that the same constant \(h_n\) is the unique minimizer of two
worst-case terminal criteria.  For smooth convex gradient descent, it uniquely
minimizes the worst-case norm of the final gradient under an initial-distance
bound.  This proves the optimal stepsize, optimal value, and uniqueness
predicted by the step-selection consequence of
Conjecture~3 in \cite{TaylorHendrickxGlineur2017}, without resolving the
conjectured worst-case curve away from its minimizer.  Through the Moreau
envelope, the same \(h_n\) is also the unique constant relaxation minimizing
the worst-case final residual of relaxed proximal point.  In both problems,
the displayed balance above equates the progress obstruction with the
oscillation obstruction.

\item \textbf{Exact proximal objective gaps for \(s\)-composable schedules and the silver schedule.}
The upper-bound mechanism is not specific to constant schedules.  For every
positive \(s\)-composable schedule, we determine the exact worst-case final
proximal objective gap and exhibit a one-dimensional absolute-value instance
attaining equality.  We then apply this general result to the original silver
schedule, whose \(s\)-composability is already known from
Lemma~6 of \cite{GrimmerShuWang2025}.  For a silver block of length
\(N=2^m-1\), we obtain
\[
    F(z_N)-F_*
    \leq
    \frac{\norm{x_0-x_*}^2}{4\lambda(1+\sqrt2)^m},
\]
with equality.  This determines the exact objective-gap constant that
Theorem~5.3 of \cite{WangMaYangZhou2025} leaves unresolved in their discussion
of the silver-schedule objective estimate.
\end{enumerate}

The two schedules illustrate two uses of the same principle. For the constant
schedule, the main difficulty is to establish \(s\)-composability; once this is
done, the resulting terminal bound matches the lower examples and identifies
the exact optimal constant choice. For the silver schedule, \(s\)-composability
is already known, and Theorem~\ref{thm:general-proximal-gap} yields the
previously unknown sharp objective guarantee. Together, these results show that
\(s\)-composability can serve as a certificate for last-iterate performance.

Section~\ref{sec:preliminaries} formally recalls \(s\)-composability and
derives the terminal consequences used throughout the paper.
Section~\ref{sec:constant} treats the constant schedule and establishes its
smooth and proximal last-iterate optimality.
Section~\ref{sec:certificate} gives a direct interpolation-certificate proof
of constant-step \(s\)-composability.
Section~\ref{sec:silver} applies the same terminal principle to the original
silver schedule and determines its exact proximal objective-gap constant.

\section{\(s\)-composability as a terminal certificate}
\label{sec:preliminaries}

We now formalize the terminal-certificate viewpoint.  The definition
below was designed to support composition, but the joint potential it contains
will be used here to extract last-iterate information.

Let \(\cF_{0,1}\) be the class of differentiable convex functions
\(f:\R^d\to\R\) whose gradient is \(1\)-Lipschitz and whose minimizer set is
nonempty.  Fix \(x_*\in\arg\min f\).  A positive schedule
\((h_0,\ldots,h_{N-1})\) generates
\begin{equation*}
 x_{k+1}=x_k-h_k\nabla f(x_k),
 \qquad k=0,\ldots,N-1.
\end{equation*}

\begin{definition}[\(s\)-composability]
\label{def:s-composable}
Following Definition~3 of \cite{GrimmerShuWang2025}, a positive schedule
\((h_0,\ldots,h_{N-1})\) is \(s\)-composable with rate \(\eta\) if
\begin{equation}
 \frac{1-\eta}{2}\norm{\nabla f(x_N)}^2
 +\frac{\eta^2}{2}\norm{x_N-x_*}^2
 +\eta(1-\eta)(f(x_N)-f_*)
 \leq \frac{\eta^2}{2}\norm{x_0-x_*}^2
 \label{eq:s-potential}
\end{equation}
for every \(f\in\cF_{0,1}\), and
\begin{equation*}
 \eta=\frac{1}{1+\sum_{k=0}^{N-1}h_k}
      =\prod_{k=0}^{N-1}(h_k-1).
\end{equation*}
\end{definition}

For the certificate construction in Section~\ref{sec:certificate}, we use the
standard smooth-convex interpolation inequality
Corollary~1 of \cite{TaylorHendrickxGlineur2017}.  Writing
\(f_i=f(x_i)\), \(g_i=\nabla f(x_i)\), and \(f_*=f(x_*)\), define
\begin{equation*}
 Q_{ij}:=2(f_i-f_j)-2\ip{g_j}{x_i-x_j}-\norm{g_i-g_j}^2\geq0.
\end{equation*}
Only this smooth-convex interpolation inequality is needed in the
certificate proof below.

We first record two schedule-independent consequences of this definition.
The first concerns the final gradient of smooth gradient descent; the second
transfers the same terminal control to relaxed proximal point.  

\begin{proposition}[Terminal consequences of \(s\)-composability]
\label{prop:terminal}
Let \((h_0,\ldots,h_{N-1})\) be \(s\)-composable with rate \(\eta\), and set
\begin{equation*}
 S:=\sum_{k=0}^{N-1}h_k=\frac{1-\eta}{\eta},
 \qquad S+1=\eta^{-1}.
\end{equation*}

\emph{(i) Smooth terminal gradient.}
For every convex \(L\)-smooth \(f\), if
\begin{equation*}
 x_{k+1}=x_k-\frac{h_k}{L}\nabla f(x_k),
\end{equation*}
then
\begin{equation*}
 \norm{\nabla f(x_N)}
 \leq
 \frac{L\norm{x_0-x_*}}{S+1}
 =L\eta\norm{x_0-x_*}.
\end{equation*}

\emph{(ii) Proximal terminal inequality.}
Let \(F:\R^d\to\R\cup\{+\infty\}\) be proper, closed, and convex with a
nonempty minimizer set, let \(\lambda>0\), and run
\begin{equation*}
 z_k=\prox_{\lambda F}(x_k),
 \qquad
 x_{k+1}=x_k+h_k(z_k-x_k).
\end{equation*}
With \(z_N=\prox_{\lambda F}(x_N)\),
\begin{equation*}
 F(z_N)-F_*
 +\frac{S+1}{\lambda}\norm{x_N-z_N}^2
 \leq
 \norm{x_0-x_*}
 \left\|\frac{x_N-z_N}{\lambda}\right\|.
\end{equation*}
Consequently,
\vspace{-2.5em}
\begin{align*}
 \left\|\frac{x_N-z_N}{\lambda}\right\|
 &\leq
 \frac{\norm{x_0-x_*}}{\lambda(S+1)},\\
 F(z_N)-F_*
 &\leq
 \frac{\norm{x_0-x_*}^2}{4\lambda(S+1)}.
\end{align*}
\end{proposition}

\begin{proof}
For part (i), scale to \(L=1\).  Divide the \(s\)-composability
potential in Definition~\ref{def:s-composable} by \(\eta^2/2\).  Since
\(S=(1-\eta)/\eta\), this gives
\begin{equation*}
 \norm{x_0-x_*}^2
 \geq
 \norm{x_N-x_*}^2
 +S(S+1)\norm{g_N}^2
 +2S(f_N-f_*),
\end{equation*}
where \(g_N=\nabla f(x_N)\).  Smooth convex interpolation between \(x_N\)
and \(x_*\) gives
\[
 \ip{g_N}{x_N-x_*}
 \geq
 f_N-f_*+\frac12\norm{g_N}^2.
\]
If \(g_N\neq0\), Cauchy--Schwarz and
the terminal potential just displayed imply
\[
 \norm{x_0-x_*}^2
 \geq
 \left(
 \frac{f_N-f_*}{\norm{g_N}}
 +\left(S+\frac12\right)\norm{g_N}
 \right)^2.
\]
Since \(f_N-f_*\geq\norm{g_N}^2/2\), the right-hand side is at least
\((S+1)^2\norm{g_N}^2\), proving the claimed terminal-gradient bound.  The
case
\(g_N=0\) is immediate.

For part (ii), let
\[
 F_\lambda(x)
 =\min_z\left\{F(z)+\frac{1}{2\lambda}\norm{z-x}^2\right\}
\]
be the Moreau envelope.  Standard Moreau-envelope identities
\cite{BauschkeCombettes2017} imply that \(\lambda F_\lambda\) is convex and
\(1\)-smooth, with
\[
 \nabla(\lambda F_\lambda)(x)=x-\prox_{\lambda F}(x).
\]
Thus the relaxed proximal iteration in part~(ii) is gradient descent with
schedule \((h_k)\) on \(\lambda F_\lambda\).  Set \(e=x_N-z_N\),
\(R=\norm{x_0-x_*}\), and \(\Delta=F(z_N)-F_*\).  Applying
the \(s\)-composability potential to \(\lambda F_\lambda\), using
\[
 \lambda F_\lambda(x_N)
 =\lambda F(z_N)+\frac12\norm e^2,
\]
and dividing by \(\eta^2/2\) yields
\begin{equation*}
 R^2
 \geq
 \norm{x_N-x_*}^2
 +S(S+2)\norm e^2
 +2S\lambda\Delta.
\end{equation*}
Because \(e/\lambda\in\partial F(z_N)\),
\[
 \ip e{z_N-x_*}\geq\lambda\Delta.
\]
For \(e\neq0\), expand \(x_N-x_*=(z_N-x_*)+e\) in the preceding proximal
terminal inequality and use the subgradient inequality.  This gives
\[
 R^2
 \geq
 \left(
 \frac{\lambda\Delta}{\norm e}
 +(S+1)\norm e
 \right)^2.
\]
Taking square roots and multiplying by \(\norm e/\lambda\) gives the joint
proximal inequality stated in the proposition; the case \(e=0\) is immediate.  Dropping the nonnegative
objective gap \(\Delta\) yields the residual bound, while maximizing the
resulting concave quadratic in \(\norm e\) yields the objective-gap bound.
\(\)
\end{proof}

The objective-gap consequence of Proposition~\ref{prop:terminal} is in fact
sharp for every positive \(s\)-composable schedule.

\begin{theorem}[Exact proximal objective gap for any \(s\)-composable schedule]
\label{thm:general-proximal-gap}
Let \((h_0,\ldots,h_{N-1})\) be a positive \(s\)-composable schedule with
rate \(\eta\).  For \(\lambda>0\), run
\[
 z_k=\prox_{\lambda F}(x_k),
 \qquad
 x_{k+1}=x_k+h_k(z_k-x_k),
\]
and set \(z_N=\prox_{\lambda F}(x_N)\).  Then
\begin{equation}
\label{eq:general-proximal-gap}
 \sup
 \frac{\lambda\bigl(F(z_N)-F_*\bigr)}
      {\norm{x_0-x_*}^2}
 =\frac{\eta}{4},
\end{equation}
where the supremum is over all dimensions, all proper closed convex
\(F\) with a nonempty minimizer set, all
\(x_*\in\arg\min F\), and all \(x_0\neq x_*\).
\end{theorem}

\begin{proof}
Set
\[
 S:=\sum_{k=0}^{N-1}h_k=\frac{1-\eta}{\eta},
 \qquad S+1=\eta^{-1}.
\]
The upper bound follows directly from Proposition~\ref{prop:terminal}:
\[
 F(z_N)-F_*
 \leq
 \frac{\norm{x_0-x_*}^2}{4\lambda(S+1)}
 =
 \frac{\eta}{4\lambda}\norm{x_0-x_*}^2.
\]

For equality, work in one dimension.  Fix \(R>0\) and take
\[
 x_*=0,\qquad x_0=R,\qquad
 F(x)=a|x|,
 \qquad
 a:=\frac{R}{2\lambda(S+1)}.
\]
Let \(T_k=\sum_{j=0}^{k-1}h_j\), with \(T_0=0\).  Since the schedule is
positive, \(0\leq T_k\leq S\).  We claim inductively that
\[
 x_k=R-\lambda aT_k,
 \qquad
 z_k=x_k-\lambda a
 \quad (0\leq k\leq N).
\]
Indeed,
\[
 x_k
 \geq R-\lambda aS
 =\frac{R(S+2)}{2(S+1)}
 >\frac{R}{2(S+1)}
 =\lambda a,
\]
so the soft-thresholding formula gives \(z_k=x_k-\lambda a\), and the
relaxed proximal update gives
\(x_{k+1}=x_k-h_k\lambda a=R-\lambda aT_{k+1}\).
At the terminal point,
\[
 z_N
 =R-\lambda a(S+1)
 =\frac{R}{2},
\]
and therefore
\[
 F(z_N)-F_*
 =a\frac{R}{2}
 =\frac{R^2}{4\lambda(S+1)}
 =\frac{\eta R^2}{4\lambda}.
\]
Thus equality is attained in \eqref{eq:general-proximal-gap}. \(\)
\end{proof}

Proposition~\ref{prop:terminal} and
Theorem~\ref{thm:general-proximal-gap} provide the terminal consequences used
for both schedules below. For the constant schedule, we first need to establish
\(s\)-composability; for the silver schedule, known \(s\)-composability allows
us to apply them directly.

\section{The constant schedule}
\label{sec:constant}

We now specialize to \(n\) repeated steps of the same size.  For a constant
schedule \((h,\ldots,h)\), the two rate identities in
Definition~\ref{def:s-composable} require
\begin{equation}
 \frac{1}{1+nh}=(h-1)^n.
 \label{eq:constant-balance}
\end{equation}
Any positive solution must lie in \((1,2)\).  Indeed, for \(h\geq2\) the
right-hand side is at least one while the left-hand side is strictly below
one.  If \(0<h<1\), the right-hand side is negative when \(n\) is odd; when
\(n\) is even,
\[
 (1-h)^n<\frac{1}{(1+h)^n}\leq\frac{1}{1+nh},
\]
where the last inequality is Bernoulli's inequality.  Thus no solution is
lost by writing \(h=1+r\) with \(r\in(0,1)\).

Define \(r_n\in(0,1)\) as the unique solution of
\begin{equation*}
 r_n^n\bigl(1+n(1+r_n)\bigr)=1,
\end{equation*}
and set
\begin{equation*}
 h_n=1+r_n,
 \qquad
 \eta_n=r_n^n=\frac{1}{1+nh_n}.
\end{equation*}
The root is unique because the left-hand side of its defining equation is
continuous and strictly increasing on \([0,1]\), equals zero at \(0\), and
exceeds one at \(1\).  Hence \(h_n\) is the unique positive constant compatible
with the two rate identities.  We will repeatedly use
\begin{equation*}
 1-\eta_n=nh_n\eta_n,
 \qquad
 r_n^{-n}=1+nh_n.
\end{equation*}

The first values are \(h_1=\sqrt{2}\) and \(h_2=3/2\).  The defining
identity also gives a simple large-horizon picture.  Indeed,
\(r_n^n(1+n(1+r_n))=1\) forces \(r_n\to1\), and therefore
\[
 h_n\to2,
 \qquad
 \eta_n=\frac{1}{1+nh_n}\sim\frac{1}{2n}.
\]
Thus the normalized constant step \(h_n\) approaches \(2\) from below as the
horizon grows (equivalently, the gradient-descent step \(h_n/L\) approaches
\(2/L\)).

\subsection{The open composability question}

The rate identities make \((h_n,\ldots,h_n)\) the only possible constant
\(s\)-composable schedule of length \(n\).  The question left open in
Example~2 of \cite{GrimmerShuWang2025} is whether this necessary candidate
satisfies the potential inequality itself for every \(n\geq3\).

\begin{theorem}[Constant-step \(s\)-composability]
\label{thm:main}
For every integer \(n\geq1\), the length-\(n\) constant schedule
\((h_n,\ldots,h_n)\) is \(s\)-composable with rate \(\eta_n\).  In
particular, every \(f\in\cF_{0,1}\), every \(x_*\in\arg\min f\), and every
trajectory \(x_{i+1}=x_i-h_n\nabla f(x_i)\) satisfy
\begin{equation*}
 \frac{1-\eta_n}{2}\norm{\nabla f(x_n)}^2
 +\frac{\eta_n^2}{2}\norm{x_n-x_*}^2
 +\eta_n(1-\eta_n)(f(x_n)-f_*)
 \leq
 \frac{\eta_n^2}{2}\norm{x_0-x_*}^2.
\end{equation*}
\end{theorem}

Theorem~\ref{thm:main} establishes the constant-step composability statement
at every finite horizon.  It also has a direct consequence for the original
purpose of \(s\)-composability: for each \(n\), the constant block
\((h_n,\ldots,h_n)\) is now available as an \(s\)-composable component in the
composition rules in \cite{GrimmerShuWang2025}.  Thus the result can be used
both in the last-iterate analysis below and in the construction of longer
composable schedules.  Before proving the theorem, we record the exact
terminal consequences that motivate the construction.

\subsection{Exact last-iterate step selection}

For normalized constant step \(\alpha>0\), define the worst-case final
gradient criterion
\begin{equation*}
 \mathcal G_n(\alpha)
 :=\sup
 \frac{\norm{\nabla f(x_n)}}{L\norm{x_0-x_*}},
\end{equation*}
where the supremum is over all dimensions and triples
\((f,x_0,x_*)\) with \(f\) convex and \(L\)-smooth,
\(x_*\in\arg\min f\), and \(x_0\neq x_*\), after \(n\) steps of gradient
descent
\[
 x_{k+1}=x_k-\frac{\alpha}{L}\nabla f(x_k).
\]

For relaxed proximal point, fix \(\lambda>0\), run
\begin{equation*}
 z_k=\prox_{\lambda F}(x_k),
 \qquad
 x_{k+1}=x_k+\alpha(z_k-x_k),
\end{equation*}
and define
\begin{equation*}
 \mathcal R_n(\alpha)
 :=\sup
 \frac{\norm{x_n-z_n}}{\norm{x_0-x_*}},
\end{equation*}
where the supremum is over all dimensions and triples
\((F,x_0,x_*)\) with \(F\) proper, closed, and convex,
\(x_*\in\arg\min F\), and \(x_0\neq x_*\).

Both criteria have the same two elementary lower obstructions.

\begin{lemma}[Progress and oscillation obstructions]
\label{lem:two-obstructions}
For every \(n\geq1\) and \(\alpha>0\),
\begin{equation*}
 \mathcal G_n(\alpha)\geq B_n(\alpha),
 \qquad
 \mathcal R_n(\alpha)\geq B_n(\alpha),
\end{equation*}
where
\begin{equation*}
 B_n(\alpha)
 :=\max\left\{
 \frac{1}{1+n\alpha},
 |1-\alpha|^n
 \right\}.
\end{equation*}
\end{lemma}

\begin{proof}
For the smooth progress branch, scale to \(L=1\), take \(x_0=R>0\), set
\(\tau=R/(1+n\alpha)\), and use the one-dimensional Huber function
\[
 \phi_\tau(x)=
 \begin{cases}
 \tfrac12x^2,& |x|\leq\tau,\\
 \tau|x|-\tfrac12\tau^2,& |x|\geq\tau.
 \end{cases}
\]
The iterates remain on the affine branch through time \(n\), so the final
gradient has magnitude \(\tau\).  The quadratic \(f(x)=Lx^2/2\) gives the
second branch.

For the proximal progress branch, take \(x_0=D>0\) and
\[
 F(x)=a|x|,
 \qquad
 a=\frac{D}{\lambda(1+n\alpha)}.
\]
Then \(x_k=D-k\alpha\lambda a\), \(x_n=\lambda a\), and \(z_n=0\),
giving \((1+n\alpha)^{-1}\).  Taking \(F=\delta_{\{0\}}\) gives
\(z_k=0\) and \(x_n=(1-\alpha)^nx_0\), yielding the oscillatory branch.
\(\)
\end{proof}

The two lower-bound branches defining \(B_n\) have a unique balancing point
on \((1,\infty)\).  Equating the progress branch
\((1+n\alpha)^{-1}\) with the oscillation branch \((\alpha-1)^n\) gives
precisely the defining equation of \(h_n\).  Theorem~\ref{thm:main} supplies
the matching upper bound at that point.

\begin{theorem}[Exact optimal constant choice for the final iterate]
\label{thm:minimax}
For every \(n\geq1\),
\begin{align*}
 \inf_{\alpha>0}\mathcal G_n(\alpha)
 &=\eta_n,
 &
 \operatorname*{argmin}_{\alpha>0}\mathcal G_n(\alpha)
 &=\{h_n\},\\
 \inf_{\alpha>0}\mathcal R_n(\alpha)
 &=\eta_n,
 &
 \operatorname*{argmin}_{\alpha>0}\mathcal R_n(\alpha)
 &=\{h_n\}.
\end{align*}
\end{theorem}

\begin{proof}
Lemma~\ref{lem:two-obstructions} gives the common lower bound
\(B_n(\alpha)\).  At \(\alpha=h_n\),
Theorem~\ref{thm:main} and Proposition~\ref{prop:terminal} give
\[
 \mathcal G_n(h_n)\leq\eta_n,
 \qquad
 \mathcal R_n(h_n)\leq\eta_n.
\]
By the definitions of \(h_n\) and \(\eta_n\), the two branches of \(B_n\) meet at \(h_n\) with
common value \(\eta_n\), so both upper bounds are exact.  On
\((1,\infty)\), the first branch is strictly decreasing and the second
strictly increasing; on \((0,1]\), the first branch is strictly larger than
its value at \(h_n>1\).  Thus \(h_n\) is the unique minimizer for both
criteria. \(\)
\end{proof}

For the smooth criterion, on the conjectured range \(0<\alpha\leq2\), the
identity \(\mathcal G_n(\alpha)=B_n(\alpha)\) is the smooth-convex specialization
of Conjecture~3 in \cite{TaylorHendrickxGlineur2017}.  We do not resolve this
identity over its full conjectured range.  Theorem~\ref{thm:minimax} proves the
associated stepsize-selection statement and, in fact, identifies the unique
minimizer over all \(\alpha>0\).

For relaxed proximal point,
Theorem~3.1 and Remark~3.1 of \cite{WangMaYangZhou2025} establish tight
constant-relaxation guarantees in their stated range.  Our result instead
optimizes the final residual over all positive constant relaxations; the
oscillatory obstruction is essential beyond that range and pins down the
unique optimum \(h_n\).

\section{Proof of constant-step \(s\)-composability}
\label{sec:certificate}

We now prove Theorem~\ref{thm:main}.  The argument has three steps.  First, a
Proposition~6 of \cite{GrimmerShuWang2025} reduces
\(s\)-composability to a single inequality along the gradient-descent history.
For a constant step, that inequality simplifies substantially.  We then prove
the simplified inequality by writing it as a nonnegative linear combination
of the smooth-convex interpolation inequalities \(Q_{ij}\geq0\).

\subsection{A sufficient history inequality}

We use only the following sufficient direction of the history-inequality
characterization.

\begin{lemma}[History inequality]
\label{lem:history}
Suppose \(\eta=(1+\sum_{i<n}h_i)^{-1}\) and every gradient-descent trajectory
generated by the schedule satisfies
\begin{align}
&\sum_{i=0}^{n-1}h_i
\left(2(f_i-f_n)+\norm{g_i}^2+2\ip{g_i}{x_0-x_i}\right)
-\norm{x_n-x_0}^2
-\frac{1-\eta}{\eta^2}\norm{g_n}^2
\geq0.
\label{eq:history}
\end{align}
Then the potential inequality in Definition~\ref{def:s-composable} holds.
\end{lemma}

\begin{proof}
Set \(g_*=0\) and add the nonnegative interpolation inequalities
\(\sum_{i<n}h_iQ_{*i}\) to the left-hand side of
\eqref{eq:history}.  Since
\(\sum_{i<n}h_ig_i=x_0-x_n\) and
\(\sum_{i<n}h_i=(1-\eta)/\eta\), the resulting expression is
\begin{align*}
\frac{2(1-\eta)}{\eta}(f_*-f_n)
+\norm{x_0-x_*}^2-\norm{x_n-x_*}^2
-\frac{1-\eta}{\eta^2}\norm{g_n}^2.
\end{align*}
It is therefore nonnegative.  Multiplying by \(\eta^2/2\) and rearranging
gives exactly \eqref{eq:s-potential}. \(\)
\end{proof}

For the remainder of the section, fix \(n\) and write
\begin{equation*}
 r=r_n,\qquad h=1+r,\qquad
 \eta=r^n=\frac{1}{1+nh}.
\end{equation*}
The constant-step trajectory satisfies
\begin{equation*}
 x_i=x_0-h\sum_{j<i}g_j.
\end{equation*}
Substituting this identity into \eqref{eq:history}, all mixed products coming
from \(2h\sum_i\ip{g_i}{x_0-x_i}\) cancel those in
\(\norm{x_n-x_0}^2\):
\begin{equation*}
2h^2\sum_{0\leq j<i<n}\ip{g_i}{g_j}
-h^2\left\|\sum_{i<n}g_i\right\|^2
=-h^2\sum_{i<n}\norm{g_i}^2.
\end{equation*}
Using \(h-h^2=-hr\), \(1-\eta=nh\eta\), and \(\eta=r^n\), Lemma~\ref{lem:history}
shows that it is enough to prove
\begin{equation}
2h\sum_{i=0}^{n-1}(f_i-f_n)
-hr\sum_{i=0}^{n-1}\norm{g_i}^2
-nhr^{-n}\norm{g_n}^2
\geq0.
\label{eq:reduced-history}
\end{equation}
We prove a stronger statement: the left-hand side of
\eqref{eq:reduced-history} is exactly a nonnegative weighted sum of the
interpolation inequalities \(Q_{ij}\geq0\).

\subsection{Two auxiliary sequences}

We now construct the nonnegative weights that will prove
\eqref{eq:reduced-history}.  The target expression contains only function
values and squared gradient norms, so after expanding a weighted sum of the
interpolation inequalities \(Q_{ij}\geq0\), every mixed term
\(\ip{g_a}{g_b}\), \(a\neq b\), must disappear.  At the same time, the
remaining coefficients must match those in \eqref{eq:reduced-history}.

Rather than choosing the \(O(n^2)\) weights one by one, we generate them from
two scalar sequences, \(A_k\) and \(D_k\).  The sequence \(A_k\) controls how
the weights vary with the earlier index, while \(D_k\) controls how they vary
with the later index.  Their recursions are chosen so that the sums appearing
in the coefficient calculations simplify, and so that all resulting weights
remain nonnegative.  The next lemma collects exactly the identities needed
for this verification.

Define
\begin{equation*}
 p_k:=\frac{r^{-2k}-1}{2},\qquad 0\leq k\leq n,
\end{equation*}
so that \(p_0=0\).  Set
\begin{equation*}
 A_{-1}=0,\qquad A_0=1,\qquad D_1=p_1,\qquad D_{n+1}=0.
\end{equation*}
For \(1\leq k<n\), define
\begin{align*}
 q_k
 &:=\frac{\bigl(1+kh+r^{-k}\bigr)\bigl(1+r^{-(k+1)}\bigr)}
 {\bigl(1+r^{-k}\bigr)\bigl(1+kh-r^{-k}\bigr)},\\
 s_k
 &:=\frac{\bigl(1+kh-r^{-k}\bigr)\bigl(r^{-(k+1)}-1\bigr)}
 {\bigl(r^{-k}-1\bigr)\bigl(1+kh+r^{-k}\bigr)},
\end{align*}
and recursively set
\begin{equation*}
 A_k=q_kA_{k-1},\qquad D_{k+1}=s_kD_k.
\end{equation*}
The following elementary identities are the only properties of these
sequences used in the certificate.

\begin{lemma}[Auxiliary-sequence identities]
\label{lem:scalar}
For every \(0\leq k<n\),
\begin{align}
&0<A_0<A_1<\cdots<A_{n-1},
\qquad
D_1>D_2>\cdots>D_n>D_{n+1}=0,
\label{eq:monotone-AD}\\
&A_kD_{k+1}=p_{k+1},
\label{eq:product-AD}\\
&\frac{\sum_{i=0}^{k-1}A_i}{A_k}
 =\frac{1+kh-r^{-k}}{h(1+r^{-(k+1)})},
\qquad
\frac{\sum_{j=k+2}^{n}D_j}{D_{k+1}}
 =\frac{1+(k+1)h-r^{-(k+1)}}{h(r^{-(k+1)}-1)}.
\label{eq:cumulative-AD}
\end{align}
Consequently,
\begin{equation}
 p_{k+1}\left(
 \frac{\sum_{j=k+2}^{n}D_j}{D_{k+1}}
 -\frac{\sum_{i=0}^{k-1}A_i}{A_k}
 \right)
 =k+1-\frac{1-r}{h}p_{k+1}.
\label{eq:cumulative-difference}
\end{equation}
Empty sums are understood as zero.
\end{lemma}

\begin{proof}
For \(1\leq k<n\), the function
\(x\mapsto\log(1+hx)/x\) is strictly decreasing on \((0,\infty)\).  Since
\(1+nh=r^{-n}\),
\begin{equation*}
 \frac{\log(1+kh)}{k}
 >\frac{\log(1+nh)}{n}
 =-\log r,
\end{equation*}
so \(1+kh>r^{-k}\).  Hence every denominator above is positive.  The formula
for \(q_k\) immediately gives \(q_k>1\), and \(s_k>0\).

It remains to check \(s_k<1\).  Put
\(a=1+kh\) and \(b=r^{-k}\).  The denominator of
\(s_k\) minus its numerator equals \(bE_k/r\), where
\begin{equation*}
 E_k=hb-(1-r)a-2r,
 \qquad
 E_1=\frac{h(1-r)^2}{r}>0,
\end{equation*}
and
\begin{equation*}
 E_{k+1}-E_k=h(1-r)(r^{-(k+1)}-1)>0.
\end{equation*}
Thus \(0<s_k<1\), proving \eqref{eq:monotone-AD}.  Moreover,
\begin{equation*}
 q_ks_k
 =\frac{r^{-2(k+1)}-1}{r^{-2k}-1}
 =\frac{p_{k+1}}{p_k}.
\end{equation*}
Since \(A_0D_1=p_1\), induction gives \eqref{eq:product-AD}.

For the first identity in \eqref{eq:cumulative-AD}, let
\(L_k=(\sum_{i<k}A_i)/A_k\).  Then \(L_0=0\) and
\(L_k=(1+L_{k-1})/q_k\); substituting the displayed formula for \(q_k\)
gives the claimed closed form by induction.  Similarly, with
\(R_k=(\sum_{j=k+2}^{n}D_j)/D_{k+1}\), one has \(R_{n-1}=0\) and
\(R_{k-1}=s_k(1+R_k)\), which gives the second closed form by backward
induction.

Finally, substitute the two formulas in \eqref{eq:cumulative-AD} into the
left-hand side of \eqref{eq:cumulative-difference}.  With
\(t=r^{-(k+1)}\), the numerator reduces to
\begin{equation*}
 2(k+1)h-(1-r)(t^2-1),
\end{equation*}
and \(p_{k+1}=(t^2-1)/2\), yielding
\eqref{eq:cumulative-difference}. \(\)
\end{proof}

\subsection{The interpolation certificate}

We now define the multipliers.  For every \(0\leq a<b\leq n\), set
\begin{align}
 \lambda_{ab}
 &:=(A_a-A_{a-1})(D_b+rD_{b+1}),
\label{eq:lambda-forward}\\
 \lambda_{ba}
 &:=(A_{a-1}+rA_a)(D_b-D_{b+1}).
\label{eq:lambda-backward}
\end{align}
By Lemma~\ref{lem:scalar}, all these multipliers are nonnegative.  A direct
expansion gives the useful identity
\begin{equation}
 \lambda_{ab}+\lambda_{ba}
 =h\bigl(A_aD_b-A_{a-1}D_{b+1}\bigr),
 \qquad a<b.
\label{eq:pair-identity}
\end{equation}
We claim that
\begin{align}
\sum_{i\ne j}\lambda_{ij}Q_{ij}
={}&2h\sum_{i=0}^{n-1}(f_i-f_n)
-hr\sum_{i=0}^{n-1}\norm{g_i}^2
-nhr^{-n}\norm{g_n}^2.
\label{eq:exact-certificate}
\end{align}
Since each \(Q_{ij}\geq0\), this identity immediately implies
\eqref{eq:reduced-history}.  It remains only to verify the coefficients.

\paragraph{Function values.}
Let
\begin{equation*}
 d_i:=\sum_{j\ne i}(\lambda_{ij}-\lambda_{ji})
\end{equation*}
be the coefficient balance attached to index \(i\).  For \(0\leq k<n\),
internal terms cancel when these balances are summed, so
\begin{equation*}
 \sum_{i=0}^k d_i
 =\sum_{a\leq k<b}(\lambda_{ab}-\lambda_{ba}).
\end{equation*}
Expanding \eqref{eq:lambda-forward}--\eqref{eq:lambda-backward} and telescoping
in \(A\) and \(D\) gives
\begin{align*}
 \sum_{i=0}^k d_i
 &=(1-r)A_kD_{k+1}
 +h\left(
 A_k\sum_{j=k+2}^{n}D_j
 -D_{k+1}\sum_{i=0}^{k-1}A_i
 \right)\\
 &=(1-r)p_{k+1}
 +hp_{k+1}\left(
 \frac{\sum_{j=k+2}^{n}D_j}{D_{k+1}}
 -\frac{\sum_{i=0}^{k-1}A_i}{A_k}
 \right)\\
 &=h(k+1),
\end{align*}
where we used \eqref{eq:product-AD} and
\eqref{eq:cumulative-difference}.  Hence \(d_i=h\) for every \(i<n\), while
\(\sum_i d_i=0\) gives \(d_n=-nh\).  Because \(Q_{ij}\) contains
\(2(f_i-f_j)\), the function-value coefficients in
\eqref{eq:exact-certificate} are therefore exactly \(2h\) for \(f_i\),
\(i<n\), and \(-2nh\) for \(f_n\).

\paragraph{Mixed gradient products.}
Fix \(a<b\).  The definitions telescope to
\begin{equation*}
 \sum_{i=0}^{a}\lambda_{ib}
 =A_a(D_b+rD_{b+1}),
 \qquad
 \sum_{i=b+1}^{n}\lambda_{ia}
 =(A_{a-1}+rA_a)D_{b+1}.
\end{equation*}
After substituting
\(x_i=x_0-h\sum_{j<i}g_j\) into \(Q_{ij}\), the coefficient of
\(\ip{g_a}{g_b}\) in \(\sum_{i\ne j}\lambda_{ij}Q_{ij}\) is
\begin{align*}
2(\lambda_{ab}+\lambda_{ba})
-2h\left(
 \sum_{i\leq a}\lambda_{ib}
 -\sum_{i>b}\lambda_{ia}
\right).
\end{align*}
The two telescoping identities and \eqref{eq:pair-identity} make this
coefficient equal to zero.  Thus every mixed product
\(\ip{g_a}{g_b}\), \(a\ne b\), cancels.

\paragraph{Squared gradient norms.}
For \(k<n\), denote the weights entering index \(k\) from smaller and larger
indices by
\begin{equation*}
 U_k:=\sum_{i<k}\lambda_{ik},
 \qquad
 V_k:=\sum_{i>k}\lambda_{ik}.
\end{equation*}
Then \(U_0=0\), and telescoping gives
\begin{equation*}
 U_k=A_{k-1}(D_k+rD_{k+1}) \quad (1\leq k<n),
 \qquad
 V_k=(A_{k-1}+rA_k)D_{k+1} \quad (0\leq k<n).
\end{equation*}
Using \eqref{eq:product-AD} at consecutive indices gives
\begin{equation}
 rV_k-U_k
 =r^2A_kD_{k+1}-A_{k-1}D_k
 =r^2p_{k+1}-p_k
 =\frac{1-r^2}{2}.
\label{eq:UV-identity}
\end{equation}
The same identity holds for \(k=0\) by direct substitution.

We already proved that \(d_k=h\) for \(k<n\).  Hence the total weight incident
to index \(k\) is \(h+2U_k+2V_k\).  The terms
\(-\norm{g_i-g_j}^2\) therefore contribute
\(-h-2U_k-2V_k\) to the coefficient of \(\norm{g_k}^2\).  The trajectory
terms in \(Q_{ij}\) contribute \(2hV_k\).  Thus, by
\eqref{eq:UV-identity}, the coefficient of \(\norm{g_k}^2\) is
\begin{equation*}
 -h-2U_k+2rV_k
 =-h+(1-r^2)
 =-hr,
 \qquad k<n.
\end{equation*}

At the terminal index,
\begin{equation*}
 U_n=\sum_{i<n}\lambda_{in}=A_{n-1}D_n=p_n,
\end{equation*}
and \(d_n=-nh\).  No trajectory term contributes to
\(\norm{g_n}^2\), so its coefficient is
\begin{equation*}
 nh-2p_n
 =nh-(r^{-2n}-1)
 =-nhr^{-n},
\end{equation*}
where the last equality uses \(r^{-n}=1+nh\).

All coefficients in \eqref{eq:exact-certificate} have now been verified.
Since every \(\lambda_{ij}\geq0\) and every \(Q_{ij}\geq0\),
\eqref{eq:reduced-history} follows.  Lemma~\ref{lem:history} gives the
potential inequality in Theorem~\ref{thm:main}, and
\(\prod_{i<n}(h-1)=r^n=\eta\) gives the second rate identity in
Definition~\ref{def:s-composable}.  This completes the proof of
Theorem~\ref{thm:main}.

\section{The silver schedule}
\label{sec:silver}

For the constant schedule, \(s\)-composability had to be established. For the
silver schedule, it is already known, so
Theorem~\ref{thm:general-proximal-gap} can be applied directly to the remaining
terminal-performance question.

Let \(\rho=1+\sqrt2\) and \(N=2^m-1\).  For the original silver schedule \cite{AltschulerParrilo2025}, Lemma~6 of
\cite{GrimmerShuWang2025} gives \(s\)-composability, and its total step length
satisfies
\begin{equation*}
 S+1=\rho^m.
\end{equation*}
Run relaxed proximal point with this schedule using the iteration from
Theorem~\ref{thm:general-proximal-gap}, and evaluate
\(z_N=\prox_{\lambda F}(x_N)\).

\begin{theorem}[Exact last-iterate objective gap for the original silver schedule]
\label{thm:silver-proximal}
For every proper closed convex \(F\) with a minimizer and every
\(x_*\in\arg\min F\),
\begin{equation*}
 F(z_N)-F_*
 \leq
 \frac{\norm{x_0-x_*}^2}{4\lambda\rho^m}.
\end{equation*}
\end{theorem}

\begin{proof}
By Lemma~6 of \cite{GrimmerShuWang2025}, the original silver schedule is
\(s\)-composable, and its total step length satisfies \(S+1=\rho^m\).
Hence its rate is \(\eta=\rho^{-m}\).  The claim follows directly from
Theorem~\ref{thm:general-proximal-gap}, which also provides a
one-dimensional absolute-value instance attaining equality. \(\)
\end{proof}

The point \(z_N\) is the same terminal proximal point studied by
Theorem~5.3 of \cite{WangMaYangZhou2025} for the original silver schedule.
Their residual estimate is tight.  In the discussion immediately following
the displayed objective bound in that theorem, they note that the
corresponding objective-gap estimate is not tight and leave its exact upper
bound open.  Theorem~\ref{thm:silver-proximal}
provides that value.  

\section{Discussion and conclusion}
\label{sec:discussion}

The two schedules highlight two distinct ways in which \(s\)-composability
can inform last-iterate analysis.  For the constant schedule, the main
challenge is existence: the rate identities single out a unique candidate,
but a nonnegative interpolation certificate is required to show that the
candidate is actually composable.  Once this is established, the terminal
bounds implied by composability meet matching lower examples and identify the
exact optimal constant choice for both the smooth final gradient and the
proximal final residual.

For the silver schedule, existence is already settled.  The same terminal
consequence can therefore be used directly, and the resulting proximal
objective estimate is exact.  This closes a previously open last-iterate
constant without modifying the schedule or constructing a new proximal
certificate.

These results show that \(s\)-composability can be used both to construct
longer schedules and to derive exact terminal guarantees. Several questions
remain.  The full
Taylor--Hendrickx--Glineur constant-step curve for the final gradient is still
open on its conjectured range away from its minimizer.  It is also natural to ask which other known
composable schedules hide similarly sharp terminal consequences, and whether
factorized certificates of the type used for the constant schedule can be
constructed for broader families of first-order or splitting methods.

\bibliographystyle{plainnat}
\bibliography{references}

\end{document}